\documentclass[12pt]{amsart}

\usepackage{amsmath,qsymbols,pifont,,upgreek,newlfont,xy}
\xyoption{all}

\newtheorem{theorem}{Theorem}[section]
\newtheorem{proposition}[theorem]{Proposition}
\newtheorem{lemma}[theorem]{Lemma}
\newtheorem{corollary}[theorem]{Corollary}
\theoremstyle{definition}
\newtheorem{definition}[theorem]{Definition}

\theoremstyle{remark}

\numberwithin{equation}{section}
\newtheorem{notation}[theorem]{Notation}

\def\bz{{\geqslant}}
\def\lz{{\leqslant}}
\def\hs#1{\hspace{#1mm}}
\def\vs#1{\vskip#1mm}
\def\pxz#1{{\prod\limits_{j = 1}^{#1}z_{_{i j}}}}
\def\pxx#1#2{{\prod\limits_{j = 1}^{#1}x_{_{#2 j}}}}
\def\pxp#1#2{{\prod\limits_{j = 1}^{#1}x'_{_{#2 j}}}}

\def\pyz#1{{\prod\limits_{j = 1}^{#1}z'_{_{i t}}}}
\def\pys#1#2#3{{\prod\limits_{{#1} = 1}^{#2}y_{_{i #3}}}}
\def\pyp#1#2#3{{\prod\limits_{{#1} = 1}^{#2}y'_{_{i #3}}}}
\def\em{{\emptyset}}

\def\r#1#2{{\Re_{#1}^{#2}}}
\def\su{{\ \subseteq \ }}
\def\la{{\lambda_e}}
\def\s#1{{\sum\limits_{i = 1}^{#1}\ }}
\def\f{{\phi_e}}
\def\les{(\lambda^{e}_{\times})^*}
\def\le{\lambda^{e}_{\times}}
\def\ls{\lambda^*_{e}}
\def\as#1{\alpha_{_+}^{#1}}
\def\ap#1{\alpha_{_{\times}}^{#1}}
\def\au#1{\alpha_{_{\cup}}^{#1}}

\begin{document}
	\title{On some strongly regular relations on hyperrings}
	
\author[S. Sh. Mousavi]{S. Sh. Mousavi}
\address[Seyed Shahin Mousavi]{Department of Pure Mathematics, Shahid Bahonar University of Kerman, Kerman, Iran}
\email{smousavi@uk.ac.ir}

	\renewcommand{\baselinestretch}{1.5}
	\oddsidemargin=1.25cm
	\evensidemargin=1.25cm
	\textwidth=15cm
	\topmargin=0cm \textheight=21.6cm
	
	\maketitle

	\begin{abstract}
In this paper first by the fact that the relation $\alpha^{*}$ is the transitive closure of two its subrelations we introduce and analyze a binary relation $\ls $ on a hyperring such that the derived ring is a unitary
ring. Next we introduce and study the notion of
  $\la$-parts in a hyperring and we characterize $\la$-parts
  in a $\la$-strong hyperring $R$. Finally we introduce a new relation $\Lambda^{*}$ such that its derived ring be a unitary
commutative ring.\\
\keywords {hyperring \and strongly regular relation \and complete parts}

\textbf{MSC(2010):} 20N20 
\end{abstract}

\section{Introduction}

A {\it hypergroupoid }  $(H, \circ )$ is a non-empty set $H$
together with a hyperoperation $\circ$ defined on $H$, that is a
mapping of $H\times H$ into the family of non-empty subsets of $H$.
If $(x, y)\in H\times H$, its image under $\circ$ is denoted by
$x\circ y$ and for simplicity by $xy$. If $A$, $B$ are non-empty
subsets of $H$ then $A\circ B$ is given by $A \circ B =\bigcup
\{xy\hspace{1mm}\mid\hspace{1mm}x\in A, y\in B\}$. $x\circ A$ is
used for $\{x\}\circ A$ (respectively $A\circ x$). A hypergroupoid $(H,
\circ)$ is called a {\it hypergroup} in the sense of~\cite{m34} if
for all $x,y,z \in H$ the following two conditions hold: (i) $x(yz)
= (xy)z$, (ii) $xH = Hx = H$, means that for any $x,y \in H$ there
exist $u,v \in H$ such that $y \in xu$  and $y \in vx$. If $(H,
\circ )$ satisfies only the first axiom, then it is called a {\it
semi-hypergroup}. An exhaustive review updated to 1992 of hypergroup
theory appears in~\cite{c93}. A recent book~\cite{cl03} contains a
wealth of applications. If $(H, \circ)$ is a semi-hypergroup (respectively hypergroup) and $R\hs1 \su \hs1 H \times H$ is an
equivalence, we set
$$A \stackrel{=}{R} B \Leftrightarrow a\hs1 R \hs1 b, \hs5 \forall a\in A, \forall b\in B,$$
for all pairs $(A , B)$ of non-empty subsets of $H$.

 The relation
$R$ is called {\it strongly regular on the left} ( {\it on the
right}) if $x \hs1 R \hs1 y \Rightarrow a\circ x \stackrel{=}{R}
a\circ y$ ( $x \hs1 R \hs1 y \Rightarrow x\circ a \stackrel{=}{R}
y\circ a$ respectively), for all $(x, y , a)\in H^3$. Moreover, $R$
is called {\it strongly regular} if it is strongly regular on the
right and on the left.

 If $(H, \circ)$ is a semi-hypergroup (respectively hypergroup) and $R$ is
 strongly regular, then the quotient $H/R$ is a semigroup (respectively group) under the operation:
 $$R(x) \otimes R(y) = R(z), \hs5 \forall z\in x \circ y,$$
 where $R(a)$ is the equivalence class of $a\in H$.

For every $n\in \mathbb{N}$,  the relation
$\beta_n$ define as follows $$\forall (x, y)\in  H^2, x \ \beta  \ y \Longleftrightarrow \exists (z_1, z_2, \cdots , z_n) \in H^n, \{x , y\} \su \prod\limits_{i = 1}^n z_i.$$
 Moreover, one puts $\beta_1 = \{(x , x) \mid x \in H \}$ and $\beta = \bigcup\limits_{n\in \mathbb{N}} \beta_n$. It is known that in every hypergroup $\beta = \beta^*$~\cite{d02}.

 In general, if $R$ is a relation of equivalence on a set $A$, then $\forall S \in P^*(A)$, we shall put $R(S) =\bigcup\limits_{x\in S} R(x)$.

 An element $e$  of a hypergroup $H$ is called an identity if $a\in a \circ e \cap e \circ a$ for all $a\in H$. An element $u$  of a hypergroup $H$ is called a scalar identity if $u \circ x = x \circ u = \{x\}$ for all $x\in H$.

 A {\it hyperring} (see ~\cite{v87}) is a triple $(R
, + , \cdot)$ which satisfies the ring-like axioms in the following
way:(i) $(R, +)$ is a hypergroup , (ii) $(R, \cdot)$ is a
semi-hypergroup, (iii) the multiplication is distributive with
respect to the hyperoperation +. The hyperrings have been studied by
many authors, for example see ~\cite{dl07,ds06,d02,ld09,v87,v94}. In~\cite{v91} Vougiouklis defines the relation
$\Gamma$ on hyperring as follows:
 $x\Gamma y$ if and only if ${x , y}\subseteq u$, where $u$ is a finite sum
 of finite products of elements of $R$, in fact there exist $n,k_{i}
 \in \mathbb{N}$  and $x_{ij} \in R$ such that $u = \sum\limits _{i=1}^{n}\prod\limits
 _{j=1}^{k_{i}}x_{ij}$. He proved that the quotient $R/\Gamma^*$, where
 $\Gamma^*$ is the transitive closure of $\Gamma$, is a ring and also
 $\Gamma^*$ is the smallest equivalent relation on $R$ such that the
 quotient $R/\Gamma^*$ is a fundamental ring. Both $\oplus$ and
 $\odot$ on $R/\Gamma^*$ are defined as follow:

 $$ \forall z\in
\Gamma^*(x) + \Gamma^*(y),\hs2 \Gamma^*(x) \oplus \Gamma^*(y) =
\Gamma^*(z);$$
  $$\forall z \in \Gamma^*(x) \cdot \Gamma^*(y),\hs2 \Gamma^*(x) \odot \Gamma^*(y) =
\Gamma^*(z) .$$
The commutativity in addition in rings can be related with the existence of the unit in multiplication. If $e$ is the unit in a ring, then for all elements $a, b$ we have
$$(a + b)(e + e) = (a + b)e + (a +b)e = a+ b + a +b,$$
$$(a + b)(e + e) = a(e + e) + b(e + e) = a + a + b + b.$$
So $a+ b + a +b = a + a + b + b$ gives $b + a = a + b$. Therefore, when we say $(R, +, .)$ is a hyperring, $(+)$ is not commutative and there is no unit in the multiplication. So The commutativity, as well as the existence of the unit, is not assumed in the fundamental ring. Of course, we know there exist many rings ($+$  is commutative) that don't have unit.

 Davvaz and Vougiouklis (2007) introduced $\alpha$ as a new strongly regular equivalence relation on a hyperring such that the set of quotients is an ordinary commutative ring.  In this paper first we show that the relation $\alpha^{*}$ produce with two subrelations and then we introduce and analyze some new binary relations $\ls $ and $\Lambda^{*}$ on hyperrings such that the derived rings are unitary and unitary commutative. We also investigate $\la$-parts on hyperrings.
  We recall the following definition from Davvaz and Vougiouklis (2007) (see~\cite{dv07}).

 \begin{definition}\label{d:alpha}
 If $R$ is a hyperring, we set
 $$\alpha_0 = \{(x , x) \mid x\in R\}$$
 and, for every integer $n\geq 1$, $\alpha_n$ is the relation defined as
 $$x\alpha_n y \Longleftrightarrow \exists (k_1, k_2, \cdots , k_n)\in \mathbb{N}^n, \exists \sigma \in S_n \hspace{5mm} \text{ and}$$
 $$\hspace{12mm} [\exists (x_{i1}, \cdots , x_{ik_i})\in R^{k_i}, \exists \sigma_i\in S_{k_i}, (i = 1, \cdots , n)] \hspace{3mm} \text{ such that}$$
 $$\hspace{15mm} x\in \s n \pxx {k_i} i  \hs5 \text{ and } \hs5 y\in \s n A_{\sigma(i)},$$
 where $A_i = \prod\limits_{j=1}^{k_i} x_{i \sigma_i(j)}$.

 Obviously, for every $n\geq 1$, the relation $\alpha_n$ are symmetric, and the relation $\alpha = \bigcup\limits_{n\geq 0}\alpha_n$ is reflexive and symmetric.
 \end{definition}

 \begin{theorem}\label{t:ring commutative}(See~\cite{dv07})

  (1) $\alpha^*$ is a strongly regular relation both on $(R , +)$ and $(R , .)$.

  (2) The quotient $R/\alpha^*$ is a commutative ring.

  (3) The relation $\alpha^*$ is the smallest equivalence relation such that the quotient $R/\alpha^*$ is a commutative ring.
 \end{theorem}

 Note that the relation $\alpha$ in the Definition~\ref{d:alpha} consists of two relations $\alpha_{+}$ and $\alpha_{\times}$ as follows:

 (1)  $ x\ \alpha_{_+} \ y \Leftrightarrow \exists n\in \mathbb{N}, \exists(k_1 , \cdots , k_n)\in \mathbb{N}^n, \exists \sigma \in S_n,$
 $$x\in  \s n \pxx {k_i} i \text { and } y\in \s n A_{\sigma(i)},$$
 where $A_i = \prod\limits_{j = 1}^{k_i} x_{ij}$. In fact, if in Definition~\ref{d:alpha} we set $\sigma_i = id$, then we obtain $\alpha_{+}$.

 (2)  $ x\ \alpha_{\times} \ y \Leftrightarrow \exists n\in \mathbb{N}, \exists(k_1 , \cdots , k_n)\in \mathbb{N}^n, \text{ and } \forall i=1 , \cdots , n$
$$ \exists(x_{i1}, \cdots , x_{ik_i}), \exists \sigma_i \in S_{k_i} \text{ such that } x\in  \s n \pxx {k_i} i \text { and } y\in \s n A_i,$$
 where $A_i = \prod\limits_{j = 1}^{k_i} x_{i \sigma_{i}(j)}$. In fact, if in Definition~\ref{d:alpha} we set $\sigma = id$, then we obtain $\alpha_{\times}$.

 \begin{corollary}\label{c:st regular alph plus}
$\alpha_{_+}^*$ and $\alpha_{_{\times}}^*$ are strongly regular relations both on $(R , +)$ and $(R , \cdot)$.
 \end{corollary}

\begin{proposition}\label{p:union}
Let $(H , \circ)$ be a semi-hypergroup and the strongly regular relations $R$ and $S$ on $H$ be given. Then $(R \cup S)^*$ (the transitive closure of $R \cup S$) is strongly regular.
\end{proposition}

\begin{proof}
Suppose that $x$ and $y$ in $H$ such that $x(R \cup S)^* y$ are given. We show that for all $a\in H$, $a \circ x  \overline{ \overline{(R \cup S)^*}} a \circ y$. For this reason let  the arbitrary elements $s\in a \circ x$ and $t\in a \circ y$ be given. By $x(R \cup S)^* y$, there exist $z_1, \cdots , z_n$ in $H$ such that
$$x (R \cup S) z_1 (R \cup S) z_2 \cdots (R \cup S)  z_n (R \cup S) y.$$
Since $R$ and $S$ are strongly regular, we have
$$a \circ x \overline{ \overline{(R \cup S)}} a \circ z_1 \overline{ \overline{(R \cup S)}} a \circ z_2 \cdots \overline{ \overline{(R \cup S)}} a \circ z_n \overline{ \overline{(R \cup S)}} a \circ y.$$
Therefore for all $i\in \{1, \cdots , n \}$ and $l_i \in a \circ z_i$, we have
$$s  (R \cup S) l_1  (R \cup S) l_2 \cdots  (R \cup S) l_n  (R \cup S) t.$$
So $s  (R \cup S)^* t$ and hence $ (R \cup S)^*$ is strongly regular to the left. Similarly we can show that  $ (R \cup S)^*$ is strongly regular to the right and the proof is complete.
\end{proof}
In the following we show that $\ap{}$ and $\as{}$ have the important role and we can produce $\alpha^{*}$ with them.
\begin{definition}\label{d:alpha union}
 Suppose that $R$ is a hyperring. Define relation $\au{}$ as follows:
$$ \au{}  = \ap{} \ \cup \ \as{}.$$
\end{definition}

 \begin{proposition}
 The quotient $R/ \au *$ is a commutative ring.
 \end{proposition}

 \begin{proof}
 We define $\oplus$ and $\otimes$ on  $R/ \au *$ as follows:
 $$\au{*}(a) \oplus \au{*}(b) = \au{*}(c), \hs2 c\in \au{*}(a) + \au{*}(b),$$
 $$\au{*}(a) \otimes \au{*}(b) = \au{*}(c), \hs2 c\in \au{*}(a) \cdot \au{*}(b).$$
 By Proposition~\ref{p:union} and Corollary~\ref{c:st regular alph plus} we have $\au{*}$ is a strongly regular both on $(R , +)$ and $(R , \cdot)$. Therefore  $\oplus$ and $\otimes$ are well defined. The associativity and distributivity on $R$ guarantee that the associativity and distributivity are valid for  $R/ \au *$. Let $x\in x_1 + x_2$ and $y\in x_2 + x_1$ be given. Put $n = 2$, $x_{11} = x_1$, $x_{21} = x_2$, $k_1 = k_2 = 1$ and $\sigma = (1 \hs2 2) \in S_2$. We have
 $$\s 2 \pxx {k_i} i = x_{11} + x_{12} = x_1 + x_2 \hs2 \text{and} \hs2 \s 2 \pxx {k_{\sigma(i)}} {\sigma(i)} = x_{_{\sigma(1) 1}} + x_{_{\sigma(2) 1}} = x_2 + x_1.$$
 Therefore $x\in \s 2 \pxx {k_i} i$ and $y\in \s 2 \pxx {k_{\sigma(i)}} {\sigma(i)}$. Thus $x\ \as \ y$ and hence $x \ \au{*} \ y$. So
 $$\au{*}(x_1) \oplus \au{*}(x_2) = \au{*}(x) = \au{*}(y) = \au{*}(x_2) \oplus \au{*}(x_1).$$
 Let $a\in x_1 \cdot x_2$ and $b\in x_2 \cdot x_1$ be given. Put $n =1$, $x_{11} = x_1$, $x_{12} = x_2$, $k_1 = 2$, $\sigma_1 = (1 \hs2 2)\in S_2$. We have
 $$\s 1 \pxx{k_1} i = x_{11} \cdot x_{12} = x_1 \cdot x_2 \hs2 \text{and} \hs2 \s 1 \prod\limits_{j=1}^{k_1}x_{_{i \sigma_i(j)}} = x_{12} \cdot x_{11} = x_2 \cdot x_1.$$
 Thus $a \in \s 1 \pxx{k_1} i$ and $b\in \s 1 \prod\limits_{j=1}^{k_1}x_{_{i \sigma_i(j)}}$. Therefore $a\ \ap \ b$ and hence $a \ \au{*} \ b$. So
 $$\au{*}(x_1) \otimes \au{*}(x_2) = \au{*}(a) = \au{*}(b) = \au{*}(x_2) \otimes \au{*}(x_1).$$
 Therefore $R/ \au *$ is a commutative ring.
 \end{proof}

 \begin{corollary}
 If $R$ is a hyperring, then $\alpha^* = \au{*}=(\ap{} \ \cup \ \as{})^{*}$.
 \end{corollary}

 \begin{proof}
 Since $\au{}\su \alpha$, $\au{*} \su \alpha^*$. By Theorem~\ref{t:ring commutative}(3), we get $\alpha^* \su \au{*}$. Therefore  $\alpha^* = \au{*}$.
 \end{proof}

 \begin{definition} (See~\cite{mad})
 A hyperring $R$ is said to be $n$-complete if for all $(k_1 , \cdots , k_n)\in \mathbb{N}^{n}$ and for all $i = 1 , \cdots , n$ and $(x_{i1} , \cdots , x_{i k_i})\in R^{k_i}$, then
 $$\Gamma\Big{(} \s n \pxx {k_i} i \Big{)} = \s n \pxx {k_i} i .$$
 \end{definition}

\section{The relation $\ls$}
 In this section by replacement a suitable relation $\le$ instead of $\alpha_{\times}$ in the Definition~\ref{d:alpha union}  we introduce the relation $\la$ to determine a new characterization of the derived hyperring.

\begin{definition}\label{d:cond-p}
 Suppose that $(R , + , .)$ is a hyperring and $e\in R$. We say that the pair $\big{(}\s m \ \pxx {k_i} i , \s m \ \pys t {k'_i} t\big{)}$ satisfying in
  condition $\mathfrak{P_{e}}$ whenever one of the following occur

  (1) $\forall i\in \{1, \cdots , m\}$, $k_i = k'_i$ and $\forall j\in \{1 , \cdots , k_i\}$, $x_{i j} = y_{i j}$;

  (2) there exist $i_1 , \cdots i_d\in\{1 , \cdots  ,m\}$ such that

   \textbullet \hs1  $\forall 1 \lz i \lz m$ such that $i \notin \{i_1 , \cdots , i_d\}$ and $\forall 1 \lz t \lz k'_i$  we have $k_i = k'_i$ and $x_{_{i t}} = y_{_{i t}}$;

   \textbullet \hs1  $\forall 1 \lz j \lz d$ there exist $p_{i_j}$ and $l_{i_j}$ in $\mathbb{N}$ such that $1 \lz p_{i_j} \lz k'_{i_j}$, $p_{i_j} \lz l_{i_j} \lz k'_{i_j}$ and $k'_{i_j} = k_{i_j} + (l_{i_j} - p_{i_j} + 1)$

 \textbullet \hs1  $\forall 1 \lz j \lz d$ we have

  \begin{equation}
y_{_{{i_j} t}} :\stackrel{\text{\tiny def}}{=} \begin{cases}
 x_{_{{i_j} t}}, &\text{if $1 \lz  t  < p_{i_j}$;}\\
 e,  &\text{if $p_{i_j} \lz t \lz l_{i_j}$;}\\
 x_{_{{i_j} (t + p_{i_j} -l_{i_j} - 1)}}, &\text{if $l_{i_j}  < t  \lz k'_{i_j}$.}
 \end{cases}\label{eq:p1}
 \end{equation}

\end{definition}

\begin{definition}\label{d:id}

 Suppose that the hyperring $R$ and $e\in R$ are given.
 For all $m \bz 1$, define:
  $$\r {m} e := \Big{\{}\big{(}\s m \ \pxx {k_i} i , \s m \ \pys t {k'_i} t\big{)}, \big{(}\s m \ \pys t {k'_i} t , \s m \ \pxx {k_i} i\big{)}  \mid \text{ the pair }$$
   $$\big{(}\s m \ \pxx {k_i} i , \s m \ \pys t {k'_i} t\big{)} \text{ satisfying in the condition } \mathfrak{P_{e}}\Big{\}};$$
and $\r{}e := \bigcup\limits_{m\bz 1} \r {m}e$.
\end{definition}

\begin{definition}

If $(R , + , \cdot)$ is a hyperring, then for every integer $m\geq 1$ we set
 $$x \hs1 (\le)_m \hs1 y \Leftrightarrow \exists (A , B)\in \r{m}e\ \  \text{such that} \hs1 x\in A\ \ \text{and} \hs1  y\in B.$$
 The relation $\le = \bigcup\limits_{m \geq 1} (\le)_m$ is reflexive and transitive.
\end{definition}

 Let $\les$ be the transitive closure of $\le$. Then we have the following.

\begin{proposition}\label{p:order}
If $R$ is a hyperring, then $(\le)_m \su (\le)_{m+1}$.
\end{proposition}

\begin{proof}
If $x \hs1 (\le)_m \hs1 y$, then there exists $(A , B) \in \r{m} e$ such that $x\in A$ and $y\in B$. Without lose the generality suppose that $A = \s m \pxx {k_i} i$ and $y = \s m \pys t {k'_i} t$. Since $(R , +)$ is a hypergroup, so there exist $u, w \in R$ such that $x_{m k_m} = u + w$. Therefore
$$\s m \pxx {k_i} i = (x_{1 1} \cdots x_{1 k_1}) + \cdots + (x_{(m-1) 1} \cdots x_{(m-1) k_{(m-1)}}) + $$
$$(x_{m 1} \cdots x_{m (k_m-1)}u) + (x_{m 1} \cdots x_{m (k_m -1)}w).$$
Let $k_{m+1} :\stackrel{\text{\tiny def}}{=} k_m$ and
$$x'_{i j} :\stackrel{\text{\tiny def}}{=}
\begin{cases}
 x_{i j}, &\text{ if $1 \lz i \lz m-1 , 1\lz j \lz k_i$;}\\
 x_{m j},  &\text{ if $i =m , 1 \lz j \lz k_m-1$;}\\
 u, &\text{ if $i =m, j = k_m$;}\\
 x_{m j}, &\text{ if $i =m+1 , 1 \lz j \lz k_{(m+1)}-1$;}\\
 w, &\text{ if $i =m+1, j = k_{m+1}$.}
  \end{cases}$$
 Also define $y'_{it}$ from $y_{it}$ as above. So
 $$ x\in \s {m+1} \pxp {k_i} i  \hs2 \text{ and } \hs2 y \in \s {m+1} \pyp t {k'_i} t,$$
 where $\Big{(}\s {m+1} \pxp {k_i} i , \s {m+1} \pyp t {k'_i} t \Big{)}\in \r {m+1} e$. Therefore $x \hs1 (\le)_{m + 1} \hs1 y$.
 \end{proof}

\begin{lemma}\label{l:str-lam pro}
$\les$ is a strongly regular relation both on $(R , +)$ and $(R , \cdot)$.
\end{lemma}

 \begin{proof}
We can see that $\les$ is an equivalence relation. In order to prove that it is strongly regular, we show that if $ x\ \le \ y $ then  $$ \begin{cases}
x + a\ \overline{\overline{\le}} \ y + a, & a + x\ \overline{\overline{\le}} \ a + y;\\
 x \cdot a\ \overline{\overline{\le}} \ y \cdot a,  & a \cdot x\ \overline{\overline{\le}} \ a \cdot y.
 \end{cases}$$
for every $a\in R$. If $x\ \le \ y$, then there exists $(A
, B)\in \r{}e$ such that $x\in A$ and $y\in B$. So there exists $m
\bz 1$ such that $A = \s m \ \pxx {k_i} i$, $B = \s m \ \pys t
{k'_i} t$ and the pair $(A, B)$ satisfying in the condition $\mathfrak{P_{e}}$.
Thus
  $$ x + a \su \big{(}\s m \ \pxx {k_i} i\big{)} + a \hs2\text{ and } \hs2 y + a \su \big{(} \s m \ \pys t {k'_i} t \big{)} + a.$$

 Now, let $k_{m + 1} = 1,\ k'_{m + 1} = 1,\ x_{_{(m + 1)1}} = a,\ y_{_{(m + 1) 1}} = a$. So
$$ x + a \su \s {m + 1} \ \pxx {k_i} i \hs2\text{ and } \hs2 y + a \su  \s {m + 1} \ \pys t {k'_i} t.$$
Therefore $\big{(}\s {m + 1 } \ \pxx {k_i} i , \s { m + 1} \ \pys t
{k'_i} t \big{)} \in \r{m + 1} e$. This implies that $(A + a, B + a)
\in \r{}e$.

 Therefore for all $u \in x + a$ and $v \in y + a$, we have $u\ \le \ v$. Thus $x + a\ \overline{\overline{\le}} \ y + a$. Similarly we can show that
  $a + x\ \overline{\overline{\le}} \ a + y$.

 Now we prove that $\les$ is a strongly regular relation on $(R , \cdot)$.

If $x\ \le \ y$, then there exists $(A , B)\in \r{}{e}$ such that $x\in A$ and $y\in B$. Therefore there exists $m \bz 1$ such that $A = \s m \ \pxx {k_i} i$,
$B = \s m \ \pys t {k'_i} t$ and the pair $\big{(}\s m \ \pxx {k_i} i , \s m \ \pys t {k'_i} t\big{)}$ satisfying in the condition $\mathfrak{P_{e}}$. Thus
$$ x \cdot a \su \big{(}\s m \ \pxx {k_i} i\big{)} \cdot a \hs2\text{ and } \hs2 y \cdot a \su \big{(} \s m \ \pys t {k'_i} t \big{)} \cdot a.$$

 Now let for all $1 \lz i \lz m$, $t_i = k_i + 1, t'_i = k'_i + 1$, $x_{i t_i} = a$ and $y_{i t'_i} = a$. So
$$ x \cdot a \su \s {m} \ \pxx {t_i} i \hs2\text{ and } \hs2 y \cdot a \su  \s {m} \ \pys t {t'_i} t,$$
and $\big{(}\s {m} \ \pxx {t_i} i , \s {m} \ \pys t {t'_i} t \big{)}
\in \r{m} e$. This implies that $(A \cdot a, B \cdot a) \in \r{}e$.
 Therefore for all $u \in x \cdot a$ and $v \in y \cdot a$, we have
 $u\ \le \ v$. Thus $x \cdot a\
 \overline{\overline{\le}} \ y \cdot a$. Similarly we can
show that $a \cdot x\ \overline{\overline{\le}} \ a . y$.

\end{proof}

\begin{definition}

 Suppose that $R$ is a hyperring. Define relation $\la$ as follows:
$$ \la  = \le \ \cup \ \alpha_{_+}.$$
\end{definition}

  Obviously, the relation $\la$ is reflexive and symmetric. Let $\ls$ be the transitive closure of $\la$. We have the following proposition.

\begin{proposition}\label{p:ring}

$\ls$ is a strongly regular relation both on $(R , +)$ and $(R , \cdot)$.
\end{proposition}

\begin{proof}

 The proof follows from Proposition~\ref{p:union}, Corollary~\ref{c:st regular alph plus} and Lemma~\ref{l:str-lam pro}.
\end{proof}

\begin{theorem}\label{t:quotient ring identity}

 The quotient $R / \ls$ is a ring with identity $\ls(e)$.
\end{theorem}

\begin{proof}

 By Proposition~\ref{p:ring} we conclude that $R / \ls$ is a ring with the following operations:
 $$\ls(a) \oplus \ls(b) = \ls(c), \hs1\text{ where}\hs1 c\in \ls(a) + \ls(b);$$
 $$\ls(a) \otimes \ls(b) = \ls(d), \hs1\text{ where}\hs1 d\in \ls(a) \cdot \ls(b).$$

  We prove that $\ls(e)$ is the identity of the ring $R / \ls$. By Definition~\ref{d:id}, for all $x\in R$ we have $(x , xe)$ and $(x , ex)$ are in $\r{1}e$. So, for all $y\in xe \cup ex$, $x\ \le \ y$ and hence $\ls(x) = \ls(y)$. Now suppose that $z\in \ls(x) \cdot \ls(e)$, so there exist $a\in \ls(x)$ and $b\in \ls(e)$ such that $z\in a b$. Thus we have
  $$a = x_1\ \la \ x_2\ \la \ ...\ \la \ x_n = x \hs1\text{ and }\hs1 b = y_1\ \la \ y_2\ \la \ ... \ \la \ y_m = e.$$
  Therefore
 $$a b = x_1 b\ \overline{\overline{\la}} \ x_2 b\ \overline{\overline{\la}} \ ...\ \overline{\overline{\la}} \ x_n b = x b \hs1\text{ and }\hs1 x b = x y_1\ \overline{\overline{\la}} \ x y_2\ \overline{\overline{\la}} \ ... \ \overline{\overline{\la}} \ x y_m = x e$$
 and so, for all $y\in x e$ we have $z \ \la \ y$. Thus $\ls(z) = \ls(y) = \ls(x)$ and hence $\ls(x) \otimes \ls(e) = \ls(x)$. Similarly we can prove $\ls(e) \otimes \ls(x) = \ls(x)$. Hence $\ls(e)$ is the identity element of the ring $R / \ls$.
\end{proof}

\begin{theorem}

The relation $\ls$ is the smallest equivalence relation such that
the quotient $R / \ls$ is a  ring with identity,
$\ls(e)$.
\end{theorem}

 \begin{proof}
 Let $\mu$ be an equivalence relation on $R$ such that $R / \mu$ is a  ring with identity $\mu(e)=\ls(e)$ and let $\xymatrix{\phi: R \ar[r] & R / \mu}$ be canonical projection. Suppose that $x\ \la \ y$. Thus we have two cases.

 {\bf case 1.} $x\ \alpha_{+} \ y$, so there exists $\sigma \in
 S_n$ such that
 $$x\in  \s n \pxx {k_i} i \hs2\text { and }\hs2 y\in \s n \pxx {k_{\sigma(i)}} {\sigma(i)}.$$
So $\phi(x) = \bigoplus\limits_{i = 1}^n\Big{(}\bigotimes\limits_{j
= 1}^{k_i}\mu(x_{_{ij}})\Big{)}$ and $\phi(y) = \bigoplus\limits_{i
= 1}^n\Big{(}\bigotimes\limits_{j =
1}^{k_{_{\sigma(i)}}}\mu(x_{_{\sigma(i)j}})\Big{)}$. By
commutativity of the group $(R / \mu , \bigoplus)$, it follows that $\phi(x) =
\phi(y)$ and hence $x \ \mu \ y$.

 {\bf case 2.} $x\ \le \ y$, so there exists $(A , B)\in \r{}e$ such that $x\in A$ and $y\in B$. Thus there exists $m \bz 1$ such that
 $A = \s m \ \pxx {k_i} i$, $B = \s m \ \pys t {k'_i} t$ and the pair $\big{(}\s m \ \pxx {k_i} i , \s m \ \pys t {k'_i} t\big{)}$ satisfying in the condition $\mathfrak{P_{e}}$. If $(A , B)$ satisfies in part (1) of Definition~\ref{d:cond-p}, then $x \ \mu \ y$. If $(A , B)$ satisfies in part (2) of Definition~\ref{d:cond-p}, then
 we have
$$\xymatrix{\bigotimes\limits_{j = 1}^{k_1}\mu(x_{1 j}) \ar@{=}[d] \bigoplus \ar @{}[r] |{...}  &  \bigoplus \bigotimes\limits_{j = 1}^{k_{i_1}}\mu(x_{i_1  j})  \ar @{}[r] |{...}  & \bigoplus \bigotimes\limits_{j = 1}^{k_{i_d}}\mu(x_{i_d j}) \bigoplus \ar @{}[r] |{...}  &\bigoplus \bigotimes\limits_{j = 1}^{k_m}\mu(x_{m j})\ar@{=}[d]\\
\bigotimes\limits_{t = 1}^{k'_1}\mu(y_{1 t}) \bigoplus \ar @{}[r] |{...}  &  \bigoplus \bigotimes\limits_{t = 1}^{k'_{i_1}}\mu(y_{i_1 t})  \ar @{}[r] |{...}  & \bigoplus \bigotimes\limits_{t = 1}^{k'_{i_d}}\mu(y_{i_d t}) \bigoplus \ar @{}[r] |{...}  &\bigoplus \bigotimes\limits_{t = 1}^{k'_m}\mu(y_{m t}).}$$
 By Definition~\ref{d:cond-p} for each $1 \lz n \lz d$, we have $\bigotimes\limits_{j = 1}^{k_{i_n}}\mu(x_{i_n j}) = \bigotimes\limits_{t = 1}^{k'_{i_n}}\mu(y_{i_n t})$. Therefore $\phi(x) = \phi(y)$ and hence $x \ \mu \ y$. Thus $ x\ \la \ y$ implies that $x \ \mu \ y$. Finally, let $x \ \ls \ y$. Since $\mu$ is transitively closed, we obtain
$$x \in \ls(y) \Rightarrow x \in \mu(y).$$
Therefore $\ls \su \mu$.

 \end{proof}

\section{Transitivity conditions of $\la$}

 In this section, we state the conditions that are equivalent to the transitivity
of the relation $\la$.
\begin{definition}\label{d:part}

 Let $M$ be a non-empty subset of hyperring $R$. We say that $M$ is a $\la$-part if

 (P1) for every $n\in \mathbb{N}$, $i = 1 , ... , n$, $\forall k_i\in \mathbb{N}$, $\forall \sigma\in S_n$, we have
 $$\s n \pxx {k_i} i \bigcap M \neq \em \Longrightarrow  \s n \pxx {k_{\sigma(i)}} {\sigma(i)} \su M ,$$

 (P2) for every $m\in \mathbb{N}$, $i = 1 , ... , m$, $\forall k_i\in \mathbb{N}$, we have

$$\s m \pxx {k_i} i \bigcap M \neq \em \Longrightarrow  \s m \ \pys t {k'_i} t \su M,$$
whenever the pair $\big{(}\s m \ \pxx {k_i} i , \s m \ \pys t {k'_i}
t\big{)}$ satisfying in the condition $\mathfrak{P_{e}}$.
\end{definition}

\begin{proposition}\label{p:part}
 Let $M$ be a non-empty subset of hyperring $R$. The following conditions are equivalent:

 (1) $M$ is a $\la$-part of $R$;

 (2) $x\in M,$ $x \la y \Rightarrow y\in M$;

 (3) $x\in M,$ $x \ls y \Rightarrow y\in M$.
\end{proposition}

\begin{proof}

 (1 $\Longrightarrow$ 2): If $(x , y) \in R^2$ is a pair such that $x\in M$ and $x \la y$, then we have two cases.

 {\bf case 1}: $x \alpha_+ y$, then
 there exists $n\in \mathbb{N}$ and there exists $\sigma\in S_n$ such that $x\in \s n \pxx {k_i} i$ and $y\in \s n \pxx {k_{\sigma(i)}} {\sigma(i)}$. Since $M$ is a $\la$-part, by Definition~\ref{d:part}(P1), we have $\s n \pxx {k_i} {\sigma(i)}\su M$ and $y\in M$.

 {\bf case 2}: $x \le y$.

Then there exists the pair $\big{(}\s m \ \pxx {k_i} i , \s m \ \pys
t {k'_i} t\big{)}$ satisfying in the condition $\mathfrak{P_{e}}$
such that $x\in \s m \ \pxx {k_i} i$
 and $y\in \s m \ \pys t {k'_i} t$. Since $M$ is a $\la$-part, by Definition~\ref{d:part}(P2), we have $\s m \ \pys t {k'_i} t \su M$ and $y\in M$.

 (2 $\Longrightarrow$ 3): Let $x\in M$ and $x \ls y$ be given. Obviously, there exist $m\in \mathbb{N}$ and $(x = w_0, w_1, ... , w_{m - 1} , w_m = y)\in R^{m + 1}$ such that $x = w_0 \la w_1 \la ... \la w_{m - 1} \la w_m = y$. Since $x\in M$ we obtain $y\in M$, by applying (2) $m$ times.

(3 $\Longrightarrow$ 1): First let $\s n \pxx {k_i} i \bigcap M
\neq\em$, and $x\in \s n \pxx {k_i} i \bigcap M$. For every $\sigma
\in S_n$ and for every  $y\in \s n \pxx {k_{\sigma(i)}} {\sigma(i)}$
we have $x \ \alpha_+  \ y$ and hence $x \ \la  \ y$. Thus $x\in M$ and $x
\ls y$ and so by (3) we obtain $y\in M$,
  where $\s n \pxx {k_{\sigma(i)}} {\sigma(i)} \su M$. Now let $\s m \pxx {k_i} i \bigcap M \neq \em$, and $x \in \s m \pxx {k_i} i \bigcap M$.
  For every $y\in \s m \ \pys t {k'_i} t$ we have $x \le y$ and hence $x \ \la \ y$. Similarly by the above we have $y\in M$ and the proof is complete.
\end{proof}

 Before proving the next theorem, we introduce the following notations.

 \begin{notation}

  For every element $x$ of a hyperring $R$, set:

  (N1) $T_n(x) = \big{\{}\s n \pxx {k_i} i \mid x \in \s n \pxx {k_i} i\big{\}};$

  (N2) $P_n^{\alpha_+}(x) = \bigcup\limits_{n \bz 1} \big{\{} \s n \pxx {k_{\sigma(i)}} {\sigma(i)} \mid \sigma\in S_n, \s n \pxx {k_i} i\in T_n(x)\big{\}};$

  (N3) $P_m^{\le}(x) = \bigcup\limits_{m \bz 1} \big{\{} \s m \ \pys t {k'_i} t \mid \s m \pxx {k_i} i\in T_m(x) \text{ and the pair} $
   $$\hs{-4}\big{(}\s m \ \pxx {k_i} i , \s m \ \pys t {k'_i} t\big{)} \text{ satisfying in the condition } \mathfrak{P_{e}}\big{\}};$$

  (N4) $P(x)  = [\bigcup\limits_{n \bz 1}P_n^{\alpha_+}(x)] \ \bigcup \ \ [ \bigcup \limits_{m \bz 1}P_m^{\le}(x)].$
 \end{notation}

 From the above notations and definitions, we obtain:

 \begin{lemma}\label{l:P}

  For every $x\in R$, $P(x) = \{y\in R \mid x \la y \}$.
 \end{lemma}

\begin{proof} It is easy to see that $P(x) \su \{y\in R \mid x \la y \}$. For every pair $(x , y)$ of elements of $R$ we have:
 \begin{align*}
x \alpha_+ y
&\Leftrightarrow
\exists n \in \mathbb{N}, \exists \sigma\in S_n : x\in \s n \pxx {k_i} i \text{ and } y\in \s n \pxx {k_{\sigma(i)}} {\sigma(i)}\\
 &\Leftrightarrow
\exists n \in \mathbb{N}, y\in P_n^{\alpha_+}(x)\\
 &\Rightarrow
y\in P(x),
\end{align*}
and
\begin{align*}
x \le y
&\Leftrightarrow
\exists m\in \mathbb{N}, x\in \s m \ \pxx {k_i} i \text{ and } y\in \s m \ \pys t {k'_i} t \text{ such that the}\\
 &
 \text{ pair } \big{(}\s m \ \pxx {k_i} i , \s m \ \pys t {k'_i} t\big{)} \text{ satisfying in the condition } \mathfrak{P_{e}}\\
&\Leftrightarrow
\exists m\in \mathbb{N}, y\in P_m(x)\\
&\Rightarrow
y\in P(x).
\end{align*}
\end{proof}

\begin{lemma}\label{l:part}

 Suppose that $R$ is a hyperring and $M$ is a $\la$-part of $R$. If $x\in M$, then $P(x) \su M$.
\end{lemma}

\begin{proof}

 The proof follows from  Lemma~\ref{l:P}.
\end{proof}

\begin{theorem}\label{t:tra}

 Suppose that $R$ is a hyperring. Then the following conditions are equivalent:

  (1) $\la$ is transitive;

  (2) for every $x\in R$, $\ls(x)$ = P(x);

  (3) for every $x\in R$, P(x) is a $\la$-part of $R$.
\end{theorem}

\begin{proof}

 (1 $\Longrightarrow$ 2): By Lemma~\ref{l:P}, for every pair $(x , y)$ of elements of $R$ we have:
$$y \in \ls(x) \Leftrightarrow x \ls y \Leftrightarrow x \la y \Leftrightarrow y\in P(x).$$

 (2 $\Longrightarrow$ 3): If $M$ is a non-empty subset of $R$, then $M$ is an $\la$-part of $R$ if and only if it is a union of equivalence classes modulo $\ls$. Particulary, every equivalence class modulo $\ls$ is a $\la$-part.

 (3 $\Longrightarrow$ 1): If $x \la y$ and $y \la z$, then by Lemma~\ref{l:P}, we have $x\in P(y)$ and $y\in P(z)$. Since $P(z)$ is a $\la$-part, by Lemma~\ref{l:part}, we obtain $P(y) \su P(z)$. Thus $x\in P(z)$ and hence by Lemma~\ref{l:P}, we have $x \la z$. Therefore $\la $ is transitive.
\end{proof}

\section{$\la$-strong hyperring and a characterization of a derived $\la$-strong hyperring}

 \begin{definition}

  Suppose that $(R , + , \cdot)$ is a hyperring and $\f$ is the canonical projection $\xymatrix{\f : R \ar[r] & R\slash\la^*}$. Define $K(R)$ and $D_e(R)$ as follows:
  $$K(R) := \f^{-1}(0_{_{R\slash\la^*}}) \hs2\text{ and }\hs2 D_{e}(R) := \f^{-1}(1_{_{R\slash\la^*}}).$$
  \end{definition}

\begin{proposition}\label{p:su}

 For a non-empty subset $M$ of a hyperring $R$ we have:

  (i) $K(R) + M  = M + K(R) = \f^{-1}(\f(M))$;

  (ii) $D_e(R) \cdot M \cup M \cdot D_e(R) \su \f^{-1}(\f(M))$.
\end{proposition}

\begin{proof}

 (i) It is easy to see that $K(R) + M \su \f^{-1}(\f(M))$. Now suppose that $x\in \f^{-1}(\f(M))$, so there exists $m\in M$, such that $\f(x) = \f(m)$ and hence $\ls(x) = \ls(m)$. Since $(R , +)$ is a hypergroup, there exists $k\in R$ such that $x\in k + m$ and so $x\in \ls(k) + \ls(m)$. Therefore $\ls(x) = \ls(k) \oplus \ls(m)$ and since $(R\slash\la^* , \oplus)$ is a group, we obtain $\ls(k) = 0_{_{R\slash\la^*}}$ and hence $k\in K(R)$. Thus $\f^{-1}(\f(M)) \su K(R) + M$.

 (ii) Since $\ls(e) = 1_{_{R\slash\la^*}}$, for every $x\in D_e(R) \cdot M$ we have $\ls(x) = \ls(m)$ for some $m\in M$ and hence $x\in \f^{-1}(\f(M))$. Thus $D_e(R) \cdot M  \su \f^{-1}(\f(M))$.
\end{proof}

\begin{definition}

 A hyperring $R$ is called $\la$-strong hyperring, whenever:

 (i) for all $x, y \in R$ if $x \ls y$, then $xe \cap ye \neq \em$ and $ex \cap ey \neq \em$;

 (ii) \{e\} is an invertible in the semi-hypergroup $(R , \cdot)$.
\end{definition}

\begin{proposition}\label{p:com}

 For a non-empty subset $M$ of a $\la$-strong hyperring $R$ we have:

  (i) $M \cdot D_e(R) = D_e(R) \cdot M = \f^{-1}(\f(M))$;

  (ii) $M$ is a $\la$-part if and only if $\f^{-1}(\f(M)) = M$.
\end{proposition}

\begin{proof}
 (i) By Proposition~\ref{p:su} it is enough to prove that $$\f^{-1}(\f(M))\su M \cdot D_e(R) \cap D_e(R) \cdot M.$$ For every $x\in \f^{-1}(\f(M))$, an element $m\in M$ exists such that $\f(x) = \f(m)$. Since $R$ is a $\la$-strong hyperring, $x e \cap m e \neq \em$. So there exists $z\in x e \cap m e$. Since $\{ e \}$ is invertible, we have $x \in z e$ and hence $x \in m e e $. Therefore $x \in M \cdot D_e(R)$ ,because $e e \su D_e(R)$. Similarly we can prove $\phi_e^{-1}(\phi_e(M)) \su D_e(R) \cdot M$.

 (ii) Suppose that $M$ is a $\la$-part and $x \in \phi_e^{-1}(\phi_e(M))$ is given. Thus there exists $m\in M$ such that $\f (x) = \f (m)$ and hence $m \ls x$, so by Proposition~\ref{p:part} we have $x\in M$. Therefore $\phi_e^{-1}(\phi_e(M)) \su M$. It is obvious that $M \su \phi_e^{-1}(\phi_e(M))$ and so $\phi_e^{-1}(\phi_e(M)) = M$. For the proof of the sufficiency suppose that $m \ls x$ and $m\in M$. Thus $\f (x) = \f (m) \in \f (M)$ and so $x \in \phi_e^{-1}(\f (M)) = M$. Therefore by Proposition~\ref{p:part} we have $M$ is an $\la$-part of $R$.
\end{proof}

\begin{theorem}
 If $R$ is a $\la$-strong hyperring, then $\la$ is transitive.
\end{theorem}

\begin{proof}
By Theorem~\ref{t:tra}, it is enough to show that for all $x\in R$, $P(x)$ is an $\la$-part of $H$. For this reason we prove that $\phi_e^{-1}(\f(P(x))) = P(x)$ and by Proposition~\ref{p:com} we have $P(x)$ is an $\la$-part of $R$.

 Suppose that $z\in \phi_e^{-1}(\phi_e(P(x)))$, so there exists $k\in P(x)$ such that $\phi_e(z) = \phi_e(k)$ and hence $\ls(z) = \ls(k)$. Since $k\in P(x)$ by Lemma~\ref{l:P} we have $x\la k$. Thus $\ls(k) = \ls(x)$ and so $\ls(z) = \ls(x)$. Since $R$ is a $\la$-strong hyperring, we have $x e \cap z e \neq \em$ and hence there exists $s \in x e \cap z e$. Therefore $x \in z e e$ and $z \in x e e$, because $\{ e \}$ is an invertible and so $z \in z e e e e$. Since $( z e e , z e e e e)\in \r {} e$, we have $x \la z$ and hence $z \in P(x)$. So we prove that $\phi_e^{-1}(\phi_e(P(x))) \su P(x)$ it is obvious $P(x) \su \phi_e^{-1}(\phi_e(P(x)))$. Therefore $\phi_e^{-1}(\phi_e(P(x))) = P(x)$ and the proof is complete.
\end{proof}

\section{The relation $\Lambda^*_{e}$ and $(\Lambda_{e})_n$-complete hyperrings  }

 \begin{definition}

 Let $R$ be a hyperring. Define relation $\Lambda_{e}$ as follows:
$$ \Lambda_{e}  = \le \ \cup \ \alpha.$$
\end{definition}

  Obviously, the relation $\Lambda_{e}$ is reflexive and symmetric. Let $\Lambda^*_{e}$ be the transitive closure of $\Lambda_{e}$.
 \begin{theorem}
        $\Lambda^*_{e}$ is a strongly regular relation both on $(R , +)$ and $(R , \cdot)$ and the quotient $R /\Lambda^*_{e}$ is a commutative ring with identity $\Lambda^*_{e}(e)$.
 \end{theorem}
 \begin{proof}
 The proof follows from Proposition~\ref{p:union}, Lemma~\ref{l:str-lam pro} and Theorem~\ref{t:quotient ring identity}. Notice that commutativity follows from Theorem~\ref{t:ring commutative}.
 \end{proof}

\begin{definition}
A hyperring $(R , + , .)$ is said to be unitary hyperring if relating to the multiplication, $(R , .)$ is a semi-hypergroup with scalar identity $e$. The element $e$ is called the unit of $R$.
\end{definition}

\begin{definition}
A hyperring  $(R , + , .)$ is said to be $(\Lambda_{e})_n$-complete if $\forall (k_1 , \cdots , k_n)\in \mathbb{N}^n$, $\forall (x_{i1} , \cdots x_{ik_i})\in R^{k_i}$, $i = 1 , \cdots n$, then
$$\Lambda_{e}\Big{(}\s n \pxx {k_i} i\Big{)} =  \s n \pys t {k'_i} t,$$
where $\Big{(} \s n \pxx {k_i} i , \s n\pys t {k'_i} t \Big{)} \in \r{n}e$.
\end{definition}

\begin{corollary}
If $R$ is a unitary hyperring with unit $e$, then $R$ is an $(\Lambda_{e})_n$-complete hyperring if and only if $R$ is an $n$-complete hyperring.
\end{corollary}

\begin{proof}
Since $(R , .)$ is a semi-hypergroup wit scalar identity $e$, then we have $\Gamma = \Lambda_{e}$.
\end{proof}

\begin{proposition}
If $(R , + , \cdot)$ is a $(\Lambda_{e})_n$-complete hyperring, then $$(\Lambda_{e})_n = \Lambda_{e}.$$
\end{proposition}

\begin{proof}
Suppose that $x \ \Lambda_{e} \ y$ thus there exists $(A , B) \in  \r {} e$ such that $x\in A$ and $y\in B$. Without lose the generality suppose that
$$A = \s m \pxx {k_i} i \hs2 \text{ and } \hs2 B = \s m \pys t {k'_i} t. $$
If $A$ and $B$ satisfying in (1) of Definition~\ref{d:cond-p}, then the proof is obvious. Now let  $A$ and $B$ satisfying in (2) of Definition~\ref{d:cond-p}.
If $m \lz n$, then by Proposition~\ref{p:order} we have $(\le)_{m} \su (\le)_{n}$. If $m > n$ since $(R , +)$ is a hypergroup, so there exists $s\in R$ such that $s \in \sum\limits_{i = n}^{m} \pxx {k_i} i$. Put
$l_i :\stackrel{\text{\tiny def}}{=} \begin{cases}
k_i, & \text { if } 1 \lz i \lz n-1;\\
1,  & \text{ if } i = n.
\end{cases}$ and
$z_{i j} :\stackrel{\text{\tiny def}}{=} \begin{cases}
x_{i j}, & \text { if } 1 \lz i \lz n-1 \text{ and } 1 \lz j \lz l_i;\\
s,  & \text{ if } i = n
\end{cases}$.
Therefore $x \in \s n \pxz{l_i}$ and since $R$ is a $(\Lambda_{e})_n$-complete hyperring, $$\Lambda_{e} (x) \su \Lambda_{e} \Big{(}  \s n \pxz{l_i} \Big{)} = \s n \pyz {l'_i},$$ where $\Big{(}  \s n \pxz{l_i}  ,   \s n \pxz{l_i} \Big{)} \in \r{n} e$. By $x \ \Lambda_{e} \ y$ we have $y\in   \Lambda_{e}(x)$ and hence $x \ (\Lambda_{e})_n \ y$.
\end{proof}

\vs{3}


\begin{thebibliography}{20}

\bibitem{c93}
P. Corsini,  {\it  Prolegomena of Hypergroup Theory}, Supplement to
Riv. Mat. Pura Appl. 2nd ed. Aviani Editor, Tricesimo (1993).

\bibitem{cl03}
P. Corsini, V. Leoreanu,  {\it  Applications of Hyperstructures
Theory}, Advanced in Mathematics. Kluwer Academic Publishers (2003).

\bibitem{dl07}
B. Davvaz, V. Leoreanu-Fotea, {\it Hyperring Theorey and
Applications}, International Academic Press, USA, (2007).

\bibitem{ds06}
B. Davvaz, A. Salasi, A realization of hyperrings, {\it Comm.
Algebra}, Vol. 34, {\bf 12} (2006), 4389-4000.

\bibitem{dv07}
B. Davvaz, T. Vougiouklis, Commutative rings obtained from hyperrings ($H_v$-rings) with $\alpha^*$-relations, {\it Comm.
Algebra}, Vol. 35, (2007), 3307-3320.

\bibitem{d02}
D. Freni, A new characterization of the derived hypergroup via
strongly regular equivalences, {\it  Communications in Algebra}, Vol.
30, {\bf 8} (2002), 3977-3989.

\bibitem{ld09}
V. Leoreanu-Fotea, B. Davvaz,  Fuzzy hyperring, {\it Fuzzy Set and
Systems}, 160 (2009), 2366--2378.

\bibitem{m34}
F. Marty, Sur uni Generalization de la Notion de Group, {\it 8th
Congress Math. Scandenaves, Stockholm, Sweden}, (1934), 45--49 .

\bibitem{mad}
S. Mirvakili, S. M. Anvariyeh, B. Davvaz, On $\alpha$-relation and transitivity condition of $\alpha$, {\it Communications in Algebra}, Vol. 36, {\bf 05} (2008), 1695--1703.

\bibitem{mad08}
S. Mirvakili, S. M. Anvariyeh, B. Davvaz, Transitivity of $\Gamma$-relation on hyperfield, {\it Bull. Mat. Soc. Sci. Math. Roumanie, Tome}, 51(99), No. 3 (2008),233--243.

\bibitem{v87}
T. Vougiouklis, Representations of hypergroups by hypermatrices,
{\it Rivista di Mat. Pure ed Appl.}, {\bf 2} (1987), 7-19.

\bibitem{v91}
T. Vougiouklis, {\it The fundamental relation in hyperrings.The
general hyperfields},  Proc. Forth Int. Congress on Algebraic
hyperstructures and Applications, 1991, Word scientific.


\bibitem{v94}
T. Vougiouklis, Hyperstructures and their representation,
{\it Hadronic press, Inc, palm Harber, USA}, {\bf 115} (1994).

\end{thebibliography}
\end{document}